\documentclass{amsart}
\usepackage{amsmath,amsfonts,amsthm}
\usepackage{amssymb,latexsym}
\usepackage{graphics}
\usepackage[colorlinks]{hyperref}

\usepackage{amssymb}

\theoremstyle{plain}
\theoremstyle{definition}
\newtheorem{theorem}{Theorem}[section]

\newtheorem{corollary}[theorem]{Corollary}

\newtheorem{example}[theorem]{Example}

\newtheorem{convention}[theorem]{Convention}
\newtheorem{remark}[theorem]{Remark}
\theoremstyle{remark}

\newtheorem*{definition}{Definition}
\numberwithin{equation}{section}

\title{On adjoint invariant classes of shift operators}
\author[F. Ayatollah Zadeh Shirazi, E. Hakimi, A. Hosseini, R. Rezavand]{Fatemah Ayatollah Zadeh Shirazi, Elaheh Hakimi \\ Arezoo Hosseini, Reza Rezavand}
\begin{document}
\begin{abstract}
In the following text we compute the adjoint of weighted generalized shift operators over Hilbert spaces.
We show for a conjugate invariant subset $A$ of $\mathbb C$,  the additive semigroup generated by
$A\cup\{0\}-$weighted generalized shifts over Hilbert space $\mathcal H$ is adjoint invariant if and only if
$\mathcal H$ is a finite dimensional Hilbert space or
$0$ is not a limit point of $A$.
\end{abstract}
\maketitle
\noindent {\small {\bf 2020 Mathematics Subject Classification:}  47B37, 47B02  \\
{\bf Keywords:}}  Adjoint operator, Hermetian operator, Hilbert space, Weighted composition operator, Weighted generalized shift
\section{Introduction}
\noindent As it has been mentioned in \cite[Chapter 6]{opt} ``A study of linear operators and adjoints is essential for a sophisticated
approach to many problems of linear vector spaces''.
Whenever a high school student tries  to compute adjoint matrix of $A=[a_{ij}]_{n\times n}$ as
$B=[b_{ij}]_{n\times n}$ with $b_{ij}=\overline{a_{ij}}$, she/he tries to compute adjoint operator.
Some related topics to adjoint operators are: self--adjoint, unitary (Hermetian), normal and $C-$normal operators \cite{cnormal}.
\\
One may consider various generalizations of
one--sided and two--sided shifts. We will use one of these generalizations which has been introduced for the first time in~\cite{note},  as  follows:
\begin{definition}
For arbitrary nonempty sets $X,\Gamma$ and self--map $\varphi:\Gamma\to\Gamma$ we call:
\[\sigma_\varphi:X^\Gamma\to X^\Gamma\:\:,\:\: (x_\alpha)_{\alpha\in\Gamma}\mapsto(x_{\varphi(\alpha)})_{\alpha\in\Gamma}\:,\]
a generalized shift. In addition, if $X$ has topological (resp. group, ring, linear vector space) structure, then
$\sigma_\varphi:X^\Gamma\to X^\Gamma$ should be continuous (resp. homomorphism) where $X^\Gamma$ is equipped with product topology (resp.
product structure).
\end{definition}
\noindent On the other hand for bounded vector $(w_n)_{n\geq1}\in\mathbb{C}^{\mathbb N}$ weighted shift
$\sigma:\ell^2\to\ell^2$ with $\sigma((x_n)_{n\geq1})=(w_nx_{n+1})_{n\geq1}$ is a well--known operator \cite{conway}.
Weighted generalized shifts can be considered  as a common generalization of generalized shifts and weighted shifts.
In this point of view weighted generalized shift has been introduced for the first time in \cite{kazem}. In fact weighted generalized shifts
are known as weighted composition operators with an old and wide background \cite{tho, jeng}.
\begin{definition}
For linear vector space $X$ over field $K$ (or semigroup $K$ acting on $X$), nonempty set $\Gamma$, self--map
$\varphi:\Gamma\to\Gamma$ and weight vector $w=(w_\alpha)_{\alpha\in\Gamma}\in K^\Gamma$ we call:
\[\sigma_{\varphi,w}:X^\Gamma\to X^\Gamma\:\:,\:\: (x_\alpha)_{\alpha\in\Gamma}\mapsto(w_\alpha x_{\varphi(\alpha)})_{\alpha\in\Gamma}\:,\]
a weighted generalized shift. We call weighted generalized shift map $\sigma_{\varphi,w}$ an $A-$weighted generalized shift if for each $\alpha$, $w_\alpha\in A$.
\end{definition}
\noindent In Hilbert space $H$ with inner product $<\:,\:>$, for operator $T:H\to H$ one may consider adjoint operator $T^*:H\to H$
satisfying $<T(x),y>\:=\:<x,T^*(y)>$ for each $x,y\in H$.
Let's recall that for each Hilbert space $H$ there exists nonzero
cardinal number $\tau$ such that $H$ and $\ell^2(\tau)$ are isomorphic as Hilbert spaces. Let's recall that
$\ell^2(\tau)$ is Hilbert space $\{x\in
\mathbb{C}^\tau:||x||_2<\infty\}$ equipped with norm $||\:||_2$
($||x||_2=(\mathop{\Sigma}\limits_{\alpha\in\tau}|x_\alpha|^2)^\frac12$ for $x=(x_\alpha)_{\alpha\in\tau}\in \mathbb{C}^\tau$) arised from
inner product:
\[<x,y>:=\mathop{\Sigma}\limits_{\alpha\in\tau}x_\alpha\overline{y_\alpha}\:\:,\:\: x=(x_\alpha)_{\alpha\in\tau},y=(y_\alpha)_{\alpha\in\tau}
\in\ell^2(\tau)\:.\]
For $\varphi:\tau\to\tau$ and $w\in \mathbb{C}^\tau$ one may consider weighted generalized shift
$\sigma_{\varphi,w}:\mathbb{C}^\tau\to \mathbb{C}^\tau$, then we may consider $\sigma_{\varphi,w}\restriction_{\ell^2(\tau)}$.
Consider the following remark regarding linear map $\sigma_{\varphi,w}:\mathbb{C}^\tau\to \mathbb{C}^\tau$.
\begin{remark}\label{salam10}
For nonzero cardinal number $\tau$, self--map $\varphi:\tau\to\tau$ and
$w=(w_\alpha)_{\alpha<\tau}\in\mathbb{C}^\tau$
the following statements are equivalent \cite{weighted}:
\begin{enumerate}
\item $\sigma_{\varphi,w}(\ell^2(\tau))\subseteq\ell^2(\tau)$,
\item $\sigma_{\varphi,w}(\ell^2(\tau))\subseteq\ell^2(\tau)$ and
	$\sigma_{\varphi,w}\restriction_{\ell^2(\tau)}:\ell^2(\tau)\to \ell^2(\tau)$ is continuous,
\item $\sup\{(\mathop{\Sigma}\limits_{\alpha\in\varphi^{-1}(\beta)}|w_\alpha|^2)^{\frac12}:\beta\in\varphi(\tau)\}<+\infty$.
\end{enumerate}
In the above case
$||\sigma_{\varphi,w}\restriction_{\ell^2(\tau)}||=\sup\{(\mathop{\Sigma}\limits_{\alpha\in\varphi^{-1}(\beta)}|w_\alpha|^2)^{\frac12}:\beta\in\varphi(\tau)\}$.
\\
Also the following statements are equivalent \cite{compact}:
\begin{enumerate}
\item $\sigma_\varphi(\ell^2(\tau))\subseteq\ell^2(\tau)$,
\item $\sigma_\varphi(\ell^2(\tau))\subseteq\ell^2(\tau)$ and
	$\sigma_\varphi\restriction_{\ell^2(\tau)}:\ell^2(\tau)\to\ell^2(\tau)$ is continuous,
\item there exists $N\geq1$ such that $card(\varphi^{-1}(\beta))\leq N$ for each $\beta\in\tau$.
\end{enumerate}
\end{remark}
\begin{convention}\label{salam}
In the following text consider nonzero cardinal number $\tau$, self--map $\varphi:\tau\to\tau$
and $w=(w_\alpha)_{\alpha\in\tau}\in {\mathbb C}^\tau$ with $\sigma_{\varphi,w}(\ell^2(\tau))\subseteq\ell^2(\tau)$.
\\
For
$\theta\in\tau$ suppose $\pi_\theta:\ell^2(\tau)\to\mathbb{C}$ is the projection map on the $\theta$th
coordinate and let $\mathsf{e}_\theta:=(\delta_\alpha^\theta)_{\alpha\in\tau}(\in\ell^2(\tau))$ where $\delta_\theta^\theta=1$ and
$\delta_\alpha^\theta=0$ for $\alpha\neq\theta$.
\\
One may equip $\ell^2(\tau)$ with the topology of pointwise convergence too (inherited from product topology of ${\mathbb C}^\tau$).
By $\ell^2_p(\tau)$ we mean $\ell^2(\tau)$ with pointwise convergence topology.
\end{convention}
\noindent The main aim of this text is to compute the adjoint of appropriate weighted generalized shifts over Hilbert spaces.
\section{Adjoint of a weighted generalized shift}
\noindent In this section we show there exist weighted generalized shift operators
$\sigma_{\eta_1,u_1}\restriction_{\ell^2(\tau)}$, $\sigma_{\eta_2,u_2}\restriction_{\ell^2(\tau)},\ldots$ on $\ell^2(\tau)$ such that
\[\sigma_{\varphi,w}\restriction_{\ell^2(\tau)}^*=
\mathop{\Sigma}\limits_{i\geq1}\sigma_{\eta_i,u_i}\restriction_{\ell^2_p(\tau)}\]
where the right hand series is a pointwise convergent series in $\ell^2_p(\tau)$ (see Convention~\ref{salam}),
such that for each $i\geq1$ we have
$\sigma_{\eta_i,u_i}(\ell^2(\tau))\subseteq\ell^2(\tau)$ and
\[u_i\in\mathop{\prod}\limits_{\alpha\in\tau}(\{\overline{w_\alpha}:\alpha\in\Gamma\}\cup\{0\})\:.\]
\begin{theorem}\label{taha10}
For conjugate invariant $A\subseteq{\mathbb C}$ (i.e., $\overline{\lambda}\in A$ for each $\lambda\in A$) the adjoint of an $A-$weighted generalized operator on Hilbert space $\ell^2(\tau)$ is a convergent series
of $A\cup\{0\}-$weighted generalized on $\ell^2_p(\tau)$.
\end{theorem}
\begin{proof}
By
$\sigma_{\varphi,w}(\ell^2(\tau))\subseteq\ell^2(\tau)$, we have
\[M:=\sup\{(\mathop{\Sigma}\limits_{\alpha\in\varphi^{-1}(\beta)}|w_\alpha|^2)^{\frac12}:\beta\in\varphi(\tau)\}<+\infty\:.\]
Since for each $\beta\in\tau$, $\mathop{\Sigma}\limits_{\alpha\in\varphi^{-1}(\beta)}|w_\alpha|^2\leq M^2<+\infty$,
 \[C_\beta:=\{\alpha\in\varphi^{-1}(\beta):w_\alpha\neq0\}\]
 is countable. \\
Choose $\psi\in\tau$. For $\beta\in\tau$ we consider $m_\beta$ and $\alpha^i_\beta$s through the following cases:
\begin{itemize}
\item[i.] $C_\beta=\varnothing$, in this case let
	\begin{itemize}
	\item[$\bullet$] $m_\beta=1$ and
 	\item[$\bullet$] $\alpha_\beta^1=\alpha_\beta^2=\cdots=\psi$,
 	\end{itemize}
\item[ii.] $C_\beta=\{\alpha_\beta^1,\ldots,\alpha_\beta^p\}$ is a finite set with $p\geq1$
elements, in this case let
	\begin{itemize}
	\item[$\bullet$] $m_\beta=p+1$ and
	\item[$\bullet$] $\alpha_\beta^i=\psi$ for $i>p$,
	\end{itemize}
\item[iii.] $C_\beta=\{\alpha_\beta^1,\alpha_\beta^2,\ldots\}$ is an infinite set with distinct
$\alpha_\beta^i$s, in this case let  $m_\beta=+\infty$.
\end{itemize}
For each $i\geq1$ consider $\eta_i:\mathop{\tau\to\tau}\limits_{\beta\mapsto\alpha_\beta^i}$ and $u_i=(u_\alpha^i)_{\alpha\in\tau}$ with:
\[u_\beta^i=\left\{\begin{array}{lc} \overline{w_{\alpha_\beta^i}} & i\leq m_\beta \:, \\ 0 & i>m_\beta\:. \end{array}
\right.\]
{\bf Claim 1.} For each $i\geq1$, $\sigma_{\eta_i,u_i}(\ell^2(\tau))\subseteq\ell^2(\tau)$.
\\
{\it Proof of Claim 1.} Consider $i\geq1$ and $\theta\in\tau$, then:
\begin{eqnarray*}
\mathop{\Sigma}\limits_{\beta\in\eta_i^{-1}(\theta)}|u^i_\beta|^2 & = &
	\mathop{\Sigma}\limits_{\theta=\eta_i(\beta)}|u^i_\beta|^2
	= \mathop{\Sigma}\limits_{\theta=\alpha_\beta^i}|u^i_\beta|^2  \\
& = &  \mathop{\Sigma}\limits_{\theta=\alpha_\beta^i,i< m_\beta}|u^i_\beta|^2
	= \mathop{\Sigma}\limits_{\theta=\alpha_\beta^i,i< m_\beta}|\overline{w_{\alpha^i_\beta}}|^2 \\
&	= &  \mathop{\Sigma}\limits_{\theta=\alpha^i_\beta, i< m_\beta}|w_{\theta}|^2 \mathop{=}\limits^{(+)}
	\left\{\begin{array}{lc}
	|w_{\theta}|^2  & \theta=\alpha_i^\beta\wedge i< m_\beta \\ 0 & otherwise
	\end{array}\right.
\end{eqnarray*}
(for (+) note that if $\beta_1,\beta_2\in\tau$ and $i_1< m_{\beta_1},i_2< m_{\beta_2}$ are such that
\linebreak $\theta=\alpha^{i_1}_{\beta_1}=\alpha^{i_2}_{\beta_2}$, then
$\theta\in C_{\beta_1}\cap C_{\beta_2}\subseteq \varphi^{-1}(\beta_1)\cap \varphi^{-1}(\beta_2)$, hence
$\beta_1=\varphi(\theta)=\beta_2$, moreover $\alpha^{i_1}_{\beta_1}=\theta=\alpha^{i_2}_{\beta_2}=\alpha^{i_2}_{\beta_1}$
and $i_1,i_2< m_{\beta_2}= m_{\beta_1}$ which leads to $i_1=i_2$)
\\
Hence $\sup\{(\mathop{\Sigma}\limits_{\beta\in\eta_i^{-1}(\theta)}|u^i_\beta|^2)^\frac12 :\theta\in\tau\}
\leq\sup\{|w_{\theta}|:\theta\in\tau\}=||w||_\infty\leq M<+\infty$ which leads to Claim 1
by Remark~\ref{salam10}.
\\
{\bf Claim 2.} $\sigma_{\varphi,w}\restriction_{\ell^2(\tau)}^*=
\mathop{\Sigma}\limits_{i\geq1}\sigma_{\eta_i,u_i}\restriction_{\ell^2_p(\tau)}$ as a pointwise convergent series in $\ell^2_p(\tau)$.
\\
{\it Proof of Claim 2.}
For each $\theta\in\tau$ and $y=(y_\alpha)_{\alpha\in\tau}
\in\ell^2(\tau)$ we have:
\begin{eqnarray*}
\pi_\theta(\sigma_{\varphi,w}\restriction^*_{\ell^2(\tau)}(y)) &= & <\sigma_{\varphi,w}\restriction^*_{\ell^2(\tau)}(y)
	,\mathsf{e}_\theta>\: =\: <y,\sigma_{\varphi,w}(\mathsf{e}_\theta)> \\
& = & <y,(w_\alpha\delta_{\varphi(\alpha)}^\theta)_{\alpha\in\tau}> \:
	= \mathop{\Sigma}\limits_{\alpha\in\tau}y_\alpha\overline{w_\alpha\delta_{\varphi(\alpha)}^\theta} \\
& = &  \mathop{\Sigma}\limits_{\alpha\in\varphi^{-1}(\theta)}y_\alpha\overline{w_\alpha}
	= \mathop{\Sigma}\limits_{\alpha=\alpha^i_\theta,i< m_\theta}y_\alpha\overline{w_\alpha} \\
& = & \mathop{\Sigma}\limits_{i< m_\theta}y_{\alpha^i_\theta}\overline{w_{\alpha^i_\theta}}
	= \mathop{\Sigma}\limits_{i< m_\theta}y_{\alpha^i_\theta}u_i^\theta
	= \mathop{\Sigma}\limits_{i< m_\theta}u_i^\theta y_{\eta_i(\theta)} \\
& = & \mathop{\Sigma}\limits_{i\geq1}u_i^\theta y_{\eta_i(\theta)}
	= \mathop{\Sigma}\limits_{i\geq1}\pi_\theta(\sigma_{\eta_i,u_i}(y)) \: (a \: convergent \: series \: in \: \mathbb{C})
\end{eqnarray*}
Claims 1 and 2 complete the proof.
\end{proof}
\begin{example}\label{salam30}
Using the same notations as in the proof of Theorem \ref{taha10} we have the following special cases:
\begin{itemize}
\item[1.] If there exists $N\geq1$ such that $m_\beta< N$ for each $\beta\in\tau$, then for each $i\geq N$, $u_i$ is zero vector and $\sigma_{\eta_i,u_i}\restriction_{\ell^2(\tau)}$ is constant zero map, thus
$\sigma_{\varphi,w}\restriction^*_{\ell^2(\tau)}=
\mathop{\Sigma}\limits_{1\leq i< N}\sigma_{\eta_i,u_i}\restriction_{\ell^2(\tau)}$ in this case.
\item[2.] If there exists $\delta>0$ such that $w_\alpha\in\{z\in\mathbb{C}:z=0\vee|z|\geq\delta\}$, then for each $\beta\in\tau$
\[(m_\beta-1)\delta^2\leq\mathop{\Sigma}\limits_{\alpha\in\varphi^{-1}(\beta)}|w_\alpha|^2\leq M^2\]
which shows $m_\beta<\left[\frac{M^2}{\delta^2}\right]+2=:N$, so $\sigma_{\varphi,w}\restriction^*_{\ell^2(\tau)}=
\mathop{\Sigma}\limits_{1\leq i< N}\sigma_{\eta_i,u_i}\restriction_{\ell^2(\tau)}$ (use (1)).
\item[3.] If $\varphi:\tau\to\tau$ is one--to--one, then for each $\beta\in \tau$ the set $\varphi^{-1}(\beta)$ has at most one element, then
$\sigma_{\varphi,w}\restriction^*_{\ell^2(\tau)}=\sigma_{\eta_1,u_1}\restriction_{\ell^2(\tau)}$.
\\
In particular,
if $\sigma_{\varphi,w}\restriction_{\ell^2(\tau)}:\ell^2(\tau)\to\ell^2(\tau)$ is bijective, $\varphi:\tau\to\tau$ is one--to--one
and $\sigma_{\varphi,w}\restriction^*_{\ell^2(\tau)}=\sigma_{\eta_1,u_1}\restriction_{\ell^2(\tau)}$ too.
\end{itemize}
\end{example}
\section{Additive semigroup generated by $A-$weighted generalized shifts}
\noindent In this section we classify all adjoint invariant subsets $A$ of $\mathbb C$ such that
additive semigroup of $A\cup\{0\}-$weighted generalized shift operators on $\ell^2(\tau)$ is adjoint invariant. Then we will
pay attention to $\{0,1\}-$weighted generalized shift operators.
\begin{theorem}\label{salam22}
For
nonempty conjugate invariant subset $A$ of $\mathbb C$  suppose
$\mathcal S$ is the additive semigroup generated by
$\{\sigma_{\phi,u}\restriction_{\ell^2(\tau)}:\phi\in\tau^\tau,u\in(A\cup\{0\})^\tau,\sigma_{\phi,u}(\ell^2(\tau))\subseteq\ell^2(\tau)\}$
(i.e., $\mathcal S$ is the additive semigroup generated by $A\cup\{0\}-$weighted generalized shift operators on $\ell^2(\tau)$). Then
the following statements are equivalent:
\begin{itemize}
\item[1.] $\mathcal S$ is adjoint invariant (i.e., $T^*\in \mathcal S$ for each $T\in\mathcal S$)
\item[2.] $\tau$ is finite or $0$ is not a limit point of $A$
\end{itemize}
\end{theorem}
\begin{proof}
``(2) $\Rightarrow$ (1)'':
Suppose $\tau$ is transfinite and zero is a limit point of $A$, then $\omega=\{0,1,2,\ldots\}\subseteq\tau$
and there exists a one--to--one sequence $\{t_n\}_{n\geq1}$ in $A$ such that
\[\forall n\geq1\:\:0<|t_{n+1}|<|t_n|<\frac1n\:.\]
Consider $\phi:\mathop{\tau\to\tau}\limits_{\alpha\mapsto1}$ and $v=(v_\alpha)_{\alpha\in\tau}\in (A\cup\{0\})^\tau$  with:
\[v_\alpha:=\left\{\begin{array}{lc} t_n & \alpha=n\in\{1,2,\ldots\}\subset\tau\:, \\ 0 & otherwise\:. \end{array}\right.\]
Then
\[\sup\{(\mathop{\Sigma}\limits_{\alpha\in\phi^{-1}(\beta)}|v_\alpha|^2)^{\frac12}:\beta\in\phi(\tau)\}
=(\mathop{\Sigma}\limits_{n\geq1}|t_n|^2)^{\frac12}\leq(\mathop{\Sigma}\limits_{n\geq1}\frac{1}{n^2})^{\frac12}<+\infty\]
and $\sigma_{\phi,v}(\ell^2(\tau))\subseteq\ell^2(\tau)$ by Remark \ref{salam10},
so $\sigma_{\phi,v}\restriction_{\ell^2(\tau)}\in\mathcal{S}$.
\\
We show $\sigma_{\phi,v}\restriction_{\ell^2(\tau)}^*\notin\mathcal{S}$,
otherwise there exist $v^1=(v^1_\alpha)_{\alpha\in\tau},\ldots,v^m=(v^m_\alpha)_{\alpha\in\tau}\in(A\cup\{0\})^\tau$ and
$\varphi_1,\ldots,\varphi_m:\tau\to\tau$ such that $\sigma_{\phi,v}\restriction_{\ell^2(\tau)}^*=
\mathop{\Sigma}\limits_{1\leq i\leq m}\sigma_{\varphi_i,v^i}\restriction_{\ell^2(\tau)}$
and $\sigma_{\varphi_j,v^j}(\ell^2(\tau))\subseteq \ell^2(\tau)$ for each $j\in\{1,\ldots,m\}$. Hence for all
$x=(x_\alpha)_{\alpha\in\tau}\in\ell^2(\tau)$ we have:
\begin{eqnarray*}
\mathop{\Sigma}\limits_{n\geq1}t_n\overline{x_n} & = & <\sigma_{\phi,v}(\mathsf{e}_1),x>
	= <\mathsf{e}_1,\sigma_{\phi,v}\restriction_{\ell^2(\tau)}^*(x)> \\
& = & \mathop{\Sigma}\limits_{1\leq i\leq m}<\mathsf{e}_1,\sigma_{\varphi_i,v^i}(x)>
	=\mathop{\Sigma}\limits_{1\leq i\leq m}\overline{v^i_1 x_{\varphi_i(1)}}
\end{eqnarray*}
Choose $p\in\{1,2,\ldots\}\setminus\{\varphi_1(1),\ldots,\varphi_m(1)\}$ and let $x=\mathsf{e}_p$, so:
\[t_p=\mathop{\Sigma}\limits_{n\geq1}t_n\overline{x_n} =\mathop{\Sigma}\limits_{1\leq i\leq m}\overline{v^i_1 x_{\varphi_i(1)}}=0\]
which is a contradiction and leads to the desired result.
\\
``(1) $\Rightarrow$ (2)'': If zero is not a limit point of $A$, then there exists $\delta>0$ such that $A\subseteq\{z\in\mathbb{C}:z=0\vee|z|\geq\delta\}$. For all $\varphi_1,\ldots,\varphi_m:\tau\to\tau$ and $w_1,\ldots,w_m\in(A\cup\{0\})^\tau$ with
$\sigma_{\varphi_i,w_i}(\ell^2(\tau))\subseteq\ell^2(\tau)$ ($i=1,\ldots,m$), by (2) in Example~\ref{salam30},
$\sigma_{\varphi_i,w_i}\restriction_{\ell^2(\tau)}^*\in\mathcal{S}$ ($i=1,\ldots,m$). Therefore
$(\mathop{\Sigma}\limits_{1\leq i\leq m}\sigma_{\varphi_i,w_i}\restriction_{\ell^2(\tau)})^*=
\mathop{\Sigma}\limits_{1\leq i\leq m}\sigma_{\varphi_i,w_i}\restriction_{\ell^2(\tau)}^*\in\mathcal{S}$.
Therefore $\mathcal{S}$ is adjoint invariant.
\\
Moreover, if $\tau=:N$ is finite, then for each $\beta\in\tau$, $\varphi^{-1}(\beta)$ has at most $N$ elements and we have the desired result
by (3) in Example~\ref{salam30}.
\end{proof}
\begin{corollary}
The collection of all generalized shift operators on $\ell^2(\tau)$ (i.e.
$\{\sigma_\phi\restriction_{\ell^2(\tau)}:\phi\in\tau^\tau,\sigma_\phi(\ell^2(\tau))\subseteq\ell^2(\tau)\}$)
is a subset of $\{0,1\}-$weighted generalized shift operators on $\ell^2(\tau)$. By Theorem~\ref{salam22}
the additive semigroup generated by $\{0,1\}-$weighted generalized shift operators on $\ell^2(\tau)$ is adjoint invariant.
In particular for generalized shift operator $\sigma_\phi\restriction_{\ell^2(\tau)}:\ell^2(\tau)\to \ell^2(\tau)$,
$\sigma_\phi\restriction_{\ell^2(\tau)}^*$ is finite summation of $\{0,1\}-$weighted generalized shift operators on $\ell^2(\tau)$
(see \cite{adjoint} too).
\end{corollary}
\section{Hermitian and unitary weighted generalized shift operators}
\noindent In Hilbert space $\mathcal H$ we call the operator $T:\mathcal{H}\to \mathcal{H}$ self--adjoint or Hermetian
if $T^*=T$. Also
we call the operator $T:\mathcal{H}\to \mathcal{H}$ unitary if it is bijective and $T^*=T^{-1}$.
\begin{theorem}[Self--adjoint weighted generalized shift operators]\label{self}
Weighted generalized shift operator $\sigma_{\varphi,w}\restriction_{\ell^2(\tau)}:\ell^2(\tau)\to\ell^2(\tau)$ is self--adjoint if and only if
for each $\theta\in\tau$:
\[w_\theta=\left\{\begin{array}{lc} 0 & \varphi^2(\theta)\neq\theta\:, \\  \overline{w_{\varphi(\theta)}} & \varphi^2(\theta)=\theta \:.
\end{array}\right.\tag{*}\]
\end{theorem}
\begin{proof}
Suppose $\sigma_{\varphi,w}\restriction_{\ell^2(\tau)}:\ell^2(\tau)\to\ell^2(\tau)$ is self--adjoint, then for each $\beta,\theta\in\tau$
we have:
\begin{eqnarray*}
w_\theta\delta^\beta_{\varphi(\theta)} & = & \pi_\theta(\sigma_{\varphi,w}(\mathsf{e}_\beta))
	=<\sigma_{\varphi,w}(\mathsf{e}_\beta),\mathsf{e}_\theta>
	= <\mathsf{e}_\beta,\sigma_{\varphi,w}\restriction_{\ell^2(\tau)}^*(\mathsf{e}_\theta)> \\
& 	= &<\mathsf{e}_\beta,\sigma_{\varphi,w}(\mathsf{e}_\theta)>
	=\overline{\pi_\beta(\sigma_{\varphi,w}(\mathsf{e}_\theta))}
	=\overline{w_\beta\delta^\theta_{\varphi(\beta)}}=\overline{w_\beta}\delta^\theta_{\varphi(\beta)}
\end{eqnarray*}
For $\beta=\varphi(\theta)$ we have:
\[w_\theta=w_\theta\delta^{\varphi(\theta)}_{\varphi(\theta)}=\overline{w_{\varphi(\theta)}}\delta^\theta_{\varphi(\varphi(\theta))}
=\left\{\begin{array}{lc} 0 & \varphi^2(\theta)\neq\theta\:, \\  \overline{w_{\varphi(\theta)}} & \varphi^2(\theta)=\theta \:.
\end{array}\right.\]
Now conversely suppose (*) holds. For each $x=(x_\alpha)_{\alpha\in\tau}\in\ell^2(\tau)$ and $\psi\in\tau$ we have
(note that if $\alpha\in\varphi^{-1}(\psi)$ and $\varphi^2(\alpha)=\alpha$, then $\alpha=\varphi(\psi)$):
\begin{eqnarray*}
\pi_\psi(\sigma_{\varphi,w}\restriction_{\ell^2(\tau)}^*(x)) & = & <\sigma_{\varphi,w}\restriction_{\ell^2(\tau)}^*(x),\mathsf{e}_\psi>
	=<x,\sigma_{\varphi,w}(\mathsf{e}_\psi)> \\
&&\\
& = & \mathop{\Sigma}\limits_{\alpha\in\tau}x_\alpha\overline{w_\alpha\delta_{\varphi(\alpha)}^\psi}
	=\mathop{\Sigma}\limits_{\alpha\in\tau}x_\alpha \overline{w_\alpha}\delta_{\varphi(\alpha)}^\psi
	= \mathop{\Sigma}\limits_{\alpha\in\varphi^{-1}(\psi)}x_\alpha\overline{ w_\alpha} \\
&&\\
& \mathop{=}\limits^{(*)}& \mathop{\Sigma}\limits_{\alpha\in\varphi^{-1}(\psi),\varphi^2(\alpha)=\alpha}x_\alpha w_{\varphi(\alpha)}
	 =	\mathop{\Sigma}\limits_{\alpha\in\varphi^{-1}(\psi),\varphi^2(\alpha)=\alpha}x_\alpha w_{\psi} \\
&& \\
& = & \mathop{\Sigma}\limits_{\alpha=\varphi(\psi)\in\varphi^{-1}(\psi),\varphi^2(\alpha)=\alpha}x_\alpha w_\psi \\
&& \\
& = & \mathop{\Sigma}\limits_{\alpha=\varphi(\psi)\in\varphi^{-1}(\psi),\varphi^2(\alpha)=\alpha}w_\psi x_{\varphi(\psi)} \\
&&\\
& \mathop{=}\limits^{(*)} & w_\psi x_{\varphi(\psi)}=\pi_\psi(\sigma_{\varphi,w}(x))
\end{eqnarray*}
By $\pi_\psi(\sigma_{\varphi,w}\restriction_{\ell^2(\tau)}^*(x)) =w_\psi x_{\varphi(\psi)}=\pi_\psi(\sigma_{\varphi,w}(x))$
for all $\psi\in\tau$, we have $\sigma_{\varphi,w}\restriction_{\ell^2(\tau)}^*(x)=\sigma_{\varphi,w}(x)$ which completes the proof.
\end{proof}
\begin{remark}[Unitary weighted generalized shift operators]\label{unit}
The following statements are equivalent \cite{isfahan}:
\\
$\bullet$ $\sigma_\varphi\restriction_{\ell^2(\tau)}:\ell^2(\tau)\to\ell^2(\tau)$ is bijective,
\\
$\bullet$ $\varphi:\tau\to\tau$ is bijective, $w_\alpha\neq0$ for all $\alpha\in\tau$,
	$\sup\{|w_\alpha|+\frac1{|w_\alpha|}:\alpha\in\tau\})<+\infty$.
\\
In particular, $\sigma_\varphi\restriction_{\ell^2(\tau)}:\ell^2(\tau)\to\ell^2(\tau)$ is an isometery (unitary (see \cite[Chapter I, Theorem 5.2 and Chapter II, Theorem 2.5]{conw} too) if and only if
$\varphi:\tau\to\tau$ is bijective and $|w_\alpha|=1$ for all $\alpha\in\tau$.
\end{remark}
\subsection*{A diagram} For transfinite $\tau$ in the class of
weighted generalized shift operators on Hilbert space
$\ell^2(\tau)$, i.e. $\mathcal{C}:=\{\sigma_{\rho,u}\restriction_{\ell^2(\tau)}:\rho\in\tau^\tau, u\in\mathbb{C}^\tau,
\sigma_{\rho,u}(\ell^2(\tau))\subseteq\ell^2(\tau)\}$, we have the following diagram:
\begin{center}
\includegraphics{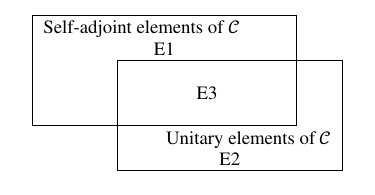}
\end{center}
Where Ei is $\sigma_{\rho_i,u_i}\restriction_{\ell^2(\tau)}$ in  the following way:
\begin{itemize}
\item $\rho_1=id_\tau:\mathop{\tau\to\tau}\limits_{\alpha\mapsto\alpha}$ and $u_1=(2)_{\alpha\in\tau}$,
\item $\rho_2=id_\tau:\mathop{\tau\to\tau}\limits_{\alpha\mapsto\alpha}$ and $u_2=(i)_{\alpha\in\tau}$,
\item $\rho_3=id_\tau:\mathop{\tau\to\tau}\limits_{\alpha\mapsto\alpha}$ and $u_1=(1)_{\alpha\in\tau}$.
\end{itemize}

\noindent \noindent {\small {\bf Fatemah Ayatollah Zadeh Shirazi}, Faculty
of Mathematics, Statistics and Computer Science, College of
Science, University of Tehran, Enghelab Ave., Tehran, Iran
\linebreak (f.a.z.shirazi@ut.ac.ir)}
\\
{\small {\bf Elaheh Hakimi}, Faculty of Mathematics, Statistics
and Computer Science, College of Science, University of Tehran,
Enghelab Ave., Tehran, Iran (elaheh.hakimi@gmail.com)}
\\
{\small {\bf Arezoo Hosseini},
Department of Mathematics, Education Farhangian University,
\linebreak
P.~O.~Box 14665--889, Tehran, Iran
(a.hosseini@cfu.ac.ir)}
\\
{\small {\bf Reza Rezavand}, School of Mathematics, Statistics
and Computer Science, College of Science, University of Tehran,
Enghelab Ave., Tehran, Iran (rezavand@ut.ac.ir)}

\end{document}